\documentclass[12pt]{amsart} 
\usepackage{fourier} 
\usepackage{fullpage} 
\usepackage{dsfont} 
\usepackage{amssymb} 
\usepackage[utf8]{inputenc} 
\usepackage[english]{babel} 

\theoremstyle{plain} 
\newtheorem{thm}{Theorem} 
\newtheorem*{lem}{Lemma} 
\newtheorem*{cor}{Corollary}

\theoremstyle{definition}

\theoremstyle{remark}

\begin{document} 
\title{Convergence Abscissas for Dirichlet Series with Multiplicative Coefficients}
\date{\today} 

\author{Ole Fredrik Brevig} 
\address{Department of Mathematical Sciences, Norwegian University of Science and Technology (NTNU), NO-7491 Trondheim, Norway} 
\email{ole.brevig@math.ntnu.no}

\author{Winston Heap} 
\address{Department of Mathematical Sciences, Norwegian University of Science and Technology (NTNU), NO-7491 Trondheim, Norway} 
\email{winstonheap@gmail.com}

\thanks{The authors are supported by Grant 227768 of the Research Council of Norway.}
\subjclass[2010]{Primary 30B50. Secondary 40A30.}
\keywords{Dirichlet series, convergence abscissas, multiplicative coefficients}
\begin{abstract}
	This note deals with the relationship between the abscissas of simple, uniform and absolute convergence for the Dirichlet series $f(s) = \sum_{n=1}^\infty a_n n^{-s}$, when the coefficients $a_n$ are either multiplicative or completely multiplicative.
\end{abstract}

\maketitle

Consider the ordinary Dirichlet series
\[f(s) = \sum_{n=1}^\infty a_n n^{-s}, \qquad s = \sigma + it.\]
A basic fact is that Dirichlet series converge in half-planes, just as power series converge in discs. However, Dirichlet series can have different types of convergence in distinct half-planes. It was H. Bohr \cite{bohr-darstellung,B13a} who first studied the relationship between the following three convergence abscissas: 
\begin{align*}
	\sigma_c(f) &= \inf \left\{\sigma \, : \, \sum_{n=1}^\infty a_n n^{-\sigma} \text{ converges}\right\} & \text{(Simple),} \\
	\sigma_b(f) &= \inf \left\{\sigma \, : \, \sum_{n=1}^\infty a_n n^{-\sigma-it} \text{ converges uniformly for } t\in\mathbb{R}\right\} &\text{(Uniform),} \\
	\sigma_a(f) &= \inf \left\{\sigma \, : \, \sum_{n=1}^\infty |a_n| n^{-\sigma} \text{ converges}\right\} &\text{(Absolute).} 
	\intertext{Clearly $\sigma_c \leq \sigma_b \leq \sigma_a$, and it is easy to deduce that $\sigma_a(f)-\sigma_c(f)\leq 1$. Under the assumption that the Dirichlet series $f$ does not converge at $s=0$, the Cauchy--Hadamard type formulas for these abscissas are:}
	\sigma_c(f) &= \limsup_{x \to \infty} \frac{1}{\log{x}}\log{\left|\sum_{n\leq x} a_n\right|}, \\
	\sigma_b(f) &= \limsup_{x \to \infty} \frac{1}{\log{x}}\log\left(\sup_{t \in \mathbb{R}}\left|\sum_{n\leq x} a_n n^{-it}\right|\right), \\
	\sigma_a(f) &= \limsup_{x \to \infty} \frac{1}{\log{x}}\log\left(\sum_{n \leq x} |a_n|\right).
\end{align*}
By choosing $a_n = \pm 1$ in a suitable manner, it is now easy to construct a Dirichlet series with $\sigma_a - \sigma_c = \alpha$, for any $\alpha\in[0,1]$. Moreover, the Cauchy--Schwarz inequality can be applied to show that $\sigma_a - \sigma_b \leq 1/2$. The fact that there are Dirichlet series with $\sigma_a - \sigma_b = \beta$ for any $\beta \in [0,1/2]$ is a result due to Bohnenblust--Hille \cite{BH31}. See \cite{BOAS} for an excellent exposition of these results, containing clear proofs using modern techniques.

The inequality used in \cite{BH31} to obtain this result was recently substantially improved \cite{BPSS,BH}, and the improved version can be used to get a precise qualitative version of the optimality of $\beta=1/2$ in view of the Cauchy--Hadamard formulas given above (see \cite{brevig}).

It is interesting to consider the difference between these abscissas when the coefficients have some added multiplicative structure (recall that $a_n$ is \emph{multiplicative} if $a_{mn} = a_m a_n$ whenever $\gcd(m,n)=1$ and is \emph{completely multiplicative} if this relationship persists for any choice of $m$ and $n$). For example, the Riemann hypothesis is equivalent to $\sigma_a-\sigma_c = 1/2$ for the series
\[1/\zeta(s) = \sum_{n=1}^\infty \mu(n)n^{-s} = \prod_{p} \left(1-p^{-s}\right),\]
where $\mu(n)$ is the Möbius function, which of course is multiplicative. 

L{\'e}vy \cite{levy} argued that any random model of the Möbius function should take into account the multiplicative nature of $\mu(n)$, and, following this, Wintner \cite{wintner} showed that the Dirichlet series represented by the Euler product 
\[\prod_{p}\left(1+\varepsilon_p p^{-s}\right)\]
has $\sigma_c = 1/2$ almost always, and concluded that ``the Riemann hypothesis is almost always true''. Here $\varepsilon_p$ denotes the Rademacher random variables which assumes the values $\pm1$ with equal probability.

Motivated by this result regarding ``typical'' behavior, we will investigate the possible values for $\sigma_a(f) - \sigma_c(f)$ and $\sigma_a(f) - \sigma_b(f)$, when the coefficients of the Dirichlet series $f$ are either multiplicative or completely multiplicative. For the first quantity, we have the following.

\begin{thm} \label{thm:simabs}
	There exists a Dirichlet series $f$ with completely multiplicative coefficients such that $\sigma_a(f)-\sigma_c(f)=\alpha$ for any $\alpha\in[0,1]$.
	\begin{proof}
		The cases $\alpha=0$ and $\alpha=1$ follow from considering the Riemann zeta function and the Dirichlet $L$-function of a non-principal character, respectively. 
		
		For $0 < \alpha < 1$, consider
		\[g_\alpha(s)=\left(1-3^{1-\alpha-s}\right)^{-1} = \sum_{k=0}^\infty 3^{(1-\alpha)k}\,3^{-ks}.\] 
		We now let $\chi$ denote the non-principal character of modulus $3$ and we consider the Dirichlet series given by the product 
		\[f(s) = g_\alpha(s)L(s,\chi).\]
		
		Clearly, $f(s)$ has completely multiplicative coefficients, since $\chi(3)=0$ and since $g_\alpha(s)$ is a geometric series. The latter fact also implies that $\sigma_c(g_\alpha)=\sigma_a(g_\alpha)=1-\alpha$, and for the $L$-function of a non-principal character we have $\sigma_c=0$ and $\sigma_a=1$. Now, the product of a conditionally convergent series and an absolutely convergent series is conditionally convergent, so we have $\sigma_c(f)\leq 1-\alpha$. This cannot be improved, since $f(1-\alpha)$ does not convergence (an infinite number of the terms have modulus $1$), so $\sigma_c(f)=1-\alpha$. 
		
		The product of two absolutely convergent series is absolutely convergent, so $\sigma_a(f)\leq1$. We let $|f|(s)$ denote the Dirichlet series where we have replaced the coefficients by their absolute values. We see that $|f|(1)$ diverges since $L(1,|\chi|)$ diverges, the coefficients of $g_\alpha$ are positive, and $g_\alpha(1)\neq0$. In conclusion, we have $\sigma_a(f)-\sigma_c(f) = 1 - (1-\alpha) = \alpha$.
	\end{proof}
\end{thm}

Of course, $g_\alpha(s)$ can be replaced by any power series in $3^{-s}$ with non-negative coefficients and $\sigma_a = 1-\alpha$ to obtain an example which is multiplicative, but not completely multiplicative.

Our next result can be considered as an example of the following scheme: A \emph{contractive} function theoretic result concerning power series, can possibly be applied \emph{multiplicatively} to obtain a similar result for ordinary Dirichlet series. A recent example of this type of result is \cite[Thm.~2]{BHS}. See also the proof of the main theorem in \cite{helson}.

\begin{thm} \label{thm:uniabs}
	Suppose that the Dirichlet series $f$ has multiplicative coefficients. Then $\sigma_a = \sigma_b$.
\end{thm}

It was H. Bohr who realized the connection between Dirichlet series and function theory in polydiscs \cite{B13b}, through the correspondence $p_j^{-s} \leftrightarrow z_j.$  Inspecting the prime factorization $n = \prod_{j} p_j^{\alpha_j}$, we associate to the integer $n$ the multi-index $\alpha(n) = (\alpha_1,\,\alpha_2,\,\ldots\,)$. The \emph{Bohr lift} of the Dirichlet series $f(s) = \sum_{n\geq1} a_n n^{-s}$ is the power series
\[\mathcal{B}f(z) = \sum_{n=1}^\infty a_n z^{\alpha(n)}.\] 
Using Kronecker's theorem \cite[Ch.~13]{hardywright} (see also \cite[Sec.~2.2]{HLS}), we may conclude that
\[\|f\|_\infty := \sup_{\sigma>0}|f(s)| = \sup_{z \in \mathbb{D}^\infty\cap c_0} |\mathcal{B}f(z)|.\]
Now, let us suppose that $f$ has multiplicative coefficients. We may then factor
\[f(s) = \prod_{j} \bigg(1 + \sum_{k=1}^\infty a_{p_j^k} p^{-ks}\bigg) = \prod_{j} f_j(s),\]
at least for $\sigma>\sigma_a$. In particular, since each prime only appears in one factor, we also obtain
\[\|f\|_\infty = \sup_{z \in \mathbb{D}^\infty\cap c_0} |\mathcal{B}f(z)| = \prod_{j} \sup_{z_j \in \mathbb{D}} |\mathcal{B}f_j(z_j)| = \prod_{j} \|f_j\|_\infty.\]
To complete the proof of Theorem~\ref{thm:uniabs}, we will require the following.
\begin{lem} \label{lem:bohrineq}
	Let $F(z) = \sum_{m\geq0} b_m z^m$ and suppose that $\sup_{z \in \mathbb{D}} |F(z)| < \infty$. Let $0\leq r < 1$. Then 
	\[\sum_{m=0}^\infty |b_m| r^m \leq C(r) \sup_{z \in \mathbb{D}} |F(z)|,\]
	where
	\[C(r) = \begin{cases}
		1, & 0\leq r \leq 1/3, \\
		1/\sqrt{1-r^2}, & 1/3 < r <1.
	\end{cases}\]
	\begin{proof}
		The first estimate is Bohr's inequality \cite{bohr1914}, the second follows from the Cauchy--Schwarz inequality, Parseval's formula and the maximum modulus principle.
	\end{proof}
\end{lem}
The contractive function theoretic result for power series mentioned earlier is that $C(r)=1$ when $0 \leq r \leq 1/3$. It should also be pointed out that the values $C(r)$ prescribed above are not optimal when $r>1/3$, and that precise estimates in this range can be found in \cite{BB}. 

\begin{proof}[Proof of Theorem~\ref{thm:uniabs}]
	Let the coefficients of $f(s) = \sum_{n\geq1} a_n n^{-s}$ be multiplicative, and fix $\varepsilon>0$. Since uniform convergence implies boundedness, we may (after a horizontal translation) assume that $\sigma_b(f)=-\varepsilon$ so that $\|f\|_\infty < \infty$. We then want to prove that under this assumption we have
	\[\sum_{n=1}^\infty |a_n| n^{-\varepsilon}<\infty,\]
	so that $\sigma_a(f)\leq\varepsilon$, and hence $\sigma_a(f) - \sigma_b(f) \leq 2\varepsilon$. Since $\varepsilon>0$ is arbitrary, $\sigma_a(f) = \sigma_b(f)$. By the discussion preceding it and the lemma, we obtain
	\begin{align*}
		\sum_{n=1}^\infty |a_n| n^{-\varepsilon} &= \prod_{p} \left(1 + \sum_{k=1}^\infty \big|a_{p^k}\big| p^{-k\varepsilon}\right) \leq \left(\prod_{p^\varepsilon < 3} \frac{\|f_p\|_{\infty}}{\sqrt{1-p^{-2\epsilon}}}\right)\left(\prod_{3 \leq p^\varepsilon<\infty} 1\cdot\|f_p\|_\infty\right) \\
		&= \left(\prod_{p^\varepsilon < 3} \frac{1}{\sqrt{1-p^{-2\epsilon}}}\right)\left(\prod_{p} \|f_p\|_\infty\right) = \left(\prod_{p^\varepsilon < 3} \frac{1}{\sqrt{1-p^{-2\epsilon}}}\right)\|f\|_\infty < \infty. \qedhere
	\end{align*}
\end{proof}

Theorem~\ref{thm:uniabs} allows us to provide a strengthening of a result of Bohr in the case of Dirichlet series with multiplicative coefficients. 
\begin{cor}
	Let $f(s) = \sum_{n\geq1} a_n n^{-s}$ have multiplicative coefficients and suppose that $f$ is somewhere convergent. If $f$ has a bounded analytic continuation to $\sigma\geq\sigma_0+\varepsilon$, for every $\varepsilon>0$, then $\sigma_a(f)=\sigma_0$. 
	\begin{proof}
		Bohr's theorem states that $\sigma_b(f)=\sigma_0$ without any assumptions on the coefficients of $f$. By Theorem~\ref{thm:uniabs}, we have $\sigma_a(f)=\sigma_b(f)=\sigma_0$.
	\end{proof}
\end{cor}

\section*{Note added in proof}
In a recent paper \cite{KP}, J.~Kaczorowski and A.~Perelli have independently proven Theorem~\ref{thm:uniabs} under the additional assumption that the Dirichlet series belongs to the Selberg class. Their methods are slightly different and do not involve analysis on the polydisc.

\bibliographystyle{amsplain} 
\bibliography{mult}

\providecommand{\bysame}{\leavevmode\hbox to3em{\hrulefill}\thinspace}
\providecommand{\MR}{\relax\ifhmode\unskip\space\fi MR }
\providecommand{\MRhref}[2]{%
  \href{http://www.ams.org/mathscinet-getitem?mr=#1}{#2}
}
\providecommand{\href}[2]{#2}
\begin{thebibliography}{10}

\bibitem{BPSS}
F.~Bayart, D.~Pellegrino, and J.~B. Seoane-Sep{\'u}lveda, \emph{The {B}ohr
  radius of the {$n$}-dimensional polydisk is equivalent to {$\sqrt{(\log
  n)/n}$}}, Adv. Math. \textbf{264} (2014), 726--746.

\bibitem{BOAS}
H.~P. Boas, \emph{The football player and the infinite series}, Notices Amer.
  Math. Soc. \textbf{44} (1997), no.~11, 1430--1435.

\bibitem{BH31}
H.~F. Bohnenblust and E.~Hille, \emph{On the absolute convergence of
  {D}irichlet series}, Ann. of Math. \textbf{32} (1931), no.~3, 600--622.

\bibitem{bohr-darstellung}
H.~Bohr, \emph{Darstellung der gleichmäßigen {K}onvergenzabszisse einer
  {Dirichletschen} reihe {$\sum_{n=1}^\infty \frac{a_n}{n^s}$} als {Funktion}
  der {Koeffizienten} der {Reihe}}, Archiv der Mathematik und Physik
  \textbf{21} (1913), no.~3, 326--330.

\bibitem{B13b}
\bysame, \emph{Über die {B}edeutung der {P}otenzreihen unendlich vieler
  {V}ariabeln in der {T}heorie der {D}irichletschen {R}eihen $\sum a_n/n^s$},
  Nachr. Akad. Wiss. Göttingen Math.-Phys. Kl. (1913), 441--488.

\bibitem{B13a}
\bysame, \emph{Über die gleichmässige {K}onvergenz {D}irichletscher
  {R}eihen}, J. Reine Angew. Math. \textbf{143} (1913), 203--211.

\bibitem{bohr1914}
\bysame, \emph{A theorem concerning power series}, Proceedings of the London
  Mathematical Society \textbf{2} (1914), no.~1, 1--5.

\bibitem{BB}
E.~Bombieri and J.~Bourgain, \emph{A remark on {Bohr's} inequality}, Int. Math.
  Res. Not. \textbf{2004} (2004), no.~80, 4307--4330.

\bibitem{BHS}
A.~Bondarenko, W.~Heap, and K.~Seip, \emph{An inequality of {Hardy--Littlewood}
  type for {D}irichlet polynomials}, J. Number Theory \textbf{150} (2015),
  no.~0, 191 -- 205.

\bibitem{brevig}
O.~F. Brevig, \emph{On the {S}idon constant for {D}irichlet polynomials}, Bull.
  Sci. Math. \textbf{138} (2014), no.~5, 656--664.

\bibitem{BH}
A.~Defant, L.~Frerick, J.~Ortega-Cerd{\`a}, M.~Ouna{\"{\i}}es, and K.~Seip,
  \emph{The {B}ohnenblust-{H}ille inequality for homogeneous polynomials is
  hypercontractive}, Ann. of Math. \textbf{174} (2011), no.~1, 485--497.

\bibitem{hardywright}
G.~H. Hardy and E.~M. Wright, \emph{An introduction to the theory of numbers},
  Oxford University Press, 1979.

\bibitem{HLS}
H.~Hedenmalm, P.~Lindqvist, and K.~Seip, \emph{A {H}ilbert space of {D}irichlet
  series and systems of dilated functions in {$L^2(0,1)$}}, Duke Math. J.
  \textbf{86} (1997), no.~1, 1--37.

\bibitem{helson}
H.~Helson, \emph{Hankel forms and sums of random variables}, Studia Math.
  \textbf{176} (2006), no.~1, 85--92.

\bibitem{KP}
J.~Kaczorowski and A.~Perelli, \emph{Some remarks on the convergence of the
  {D}irichlet series of {L}-functions}, arXiv:1506.07630 (2015).

\bibitem{levy}
M.~L{\'e}vy, \emph{Sur les s{\'e}ries dont les termes sont des variables
  {\'e}ventuelles ind{\'e}pendantes}, Studia Math. \textbf{3} (1931), no.~1,
  119--155.

\bibitem{wintner}
A.~Wintner, \emph{Random factorizations and {Riemann’s} hypothesis}, Duke
  Math. J. \textbf{11} (1944), no.~2, 267--275.

\end{thebibliography}
\end{document}